\documentclass[12pt]{amsart}
\usepackage[margin=1.35in]{geometry}
\usepackage{amscd,amsmath,amsxtra,amsthm,amssymb,stmaryrd,xr,mathrsfs,mathtools,enumerate,commath, comment}
\usepackage{stmaryrd}
\usepackage{xcolor}
\usepackage{commath}
\usepackage{comment}
\usepackage{tikz-cd}
\usepackage{longtable} 
\usepackage{pdflscape} 
\usepackage{booktabs}
\usepackage{hyperref}
\definecolor{vegasgold}{rgb}{0.77, 0.7, 0.35}
\definecolor{darkgoldenrod}{rgb}{0.72, 0.53, 0.04}
\definecolor{gold(metallic)}{rgb}{0.83, 0.69, 0.22}
\hypersetup{
 colorlinks=true,
 linkcolor=darkgoldenrod,
 filecolor=brown,      
 urlcolor=gold(metallic),
 citecolor=darkgoldenrod,
 }
\newtheorem{lthm}{Theorem}

\usepackage[all,cmtip]{xy}

\DeclareFontFamily{U}{wncy}{}
\DeclareFontShape{U}{wncy}{m}{n}{<->wncyr10}{}
\DeclareSymbolFont{mcy}{U}{wncy}{m}{n}
\DeclareMathSymbol{\Sh}{\mathord}{mcy}{"58}
\usepackage[T2A,T1]{fontenc}
\usepackage[OT2,T1]{fontenc}

\newtheorem{theorem}{Theorem}[section]
\newtheorem{lemma}[theorem]{Lemma}
\newtheorem{ass}[theorem]{Assumption}
\newtheorem*{theorem*}{Theorem}
\newtheorem*{ass*}{Assumption}
\newtheorem{definition}[theorem]{Definition}
\newtheorem{corollary}[theorem]{Corollary}
\newtheorem{remark}[theorem]{Remark}

\newtheorem{proposition}[theorem]{Proposition}

\newcommand{\cK}{\mathcal{K}}

\newcommand{\Z}{\mathbb{Z}}

\newcommand{\Q}{\mathbb{Q}}
\newcommand{\F}{\mathbb{F}}

\newcommand{\corank}{\mathrm{corank}}

\newcommand{\op}[1]{\operatorname{#1}}

\newcommand\mtx[4] { \left( {\begin{array}{cc}
 #1 & #2 \\
 #3 & #4 \\
 \end{array} } \right)}

\numberwithin{equation}{section}

\begin{document}

\title[Kida's formula for elliptic curves with additive reduction]{An analogue of Kida's formula for elliptic curves with additive reduction}

\author[A.~Ray]{Anwesh Ray $^\ddag$}
\thanks{$^\ddag$(Corresponding author) email: anwesh@cmi.ac.in}
\address[Ray]{Chennai Mathematical Institute, H1, SIPCOT IT Park, Kelambakkam, Siruseri, Tamil Nadu 603103, India}
\email{anwesh@cmi.ac.in}

\author[P.~Shingavekar]{Pratiksha Shingavekar}
\address[Shingavekar]{Chennai Mathematical Institute, H1, SIPCOT IT Park, Kelambakkam, Siruseri, Tamil Nadu 603103, India}
\email{pshingavekar@gmail.com}

\keywords{rank stability, Iwasawa theory of elliptic curves, additive reduction, Kida's formula}
\subjclass[2020]{11R23, 11G05 (primary) 11G07 (secondary)}

\maketitle

\begin{abstract}
We study the Iwasawa theory of $p$-primary Selmer groups of elliptic curves $E$ over a number field $K$. Assume that $E$ has additive reduction at the primes of $K$ above $p$. In this context, we prove that the Iwasawa invariants satisfy an analogue of the Riemann--Hurwitz formula. This generalizes a result of Hachimori and Matsuno. We apply our results to study rank stability questions for elliptic curves in prime cyclic extensions of $\mathbb{Q}$. These extensions are ordered by their absolute discriminant and we prove an asymptotic lower bound for the density of extensions in which the Iwasawa invariants as well as the rank of the elliptic curve is stable.
\end{abstract}

\section{Introduction}

\subsection{Motivation and historical context}
\par Let $p$ be a prime number and $\Z_p$ be the ring of $p$-adic integers. Iwasawa \cite{Iwasawareference} studied growth patterns of $p$-primary parts of class numbers in certain infinite abelian Galois towers of number fields. Let $F$ be a number field. Setting $\mu_{p^\infty}\subset \bar{\Q}$ to be the $p$-primary roots of unity, we let $F(\mu_{p^\infty})$ denote the Galois extension of $F$ that is generated by $\mu_{p^\infty}$. There is a unique $\Z_p$-extension $F_{\op{cyc}}/F$ which is contained in $F(\mu_{p^\infty})$. This is called the \emph{cyclotomic $\Z_p$-extension} of $F$ and the Galois group $\op{Gal}(F_{\op{cyc}}/F)$ is topologically isomorphic to $\Z_p$. For a natural number $n$, we define $F_n$ to be the subfield of $F_{\op{cyc}}$ such that $\op{Gal}(F_n/F) \simeq \Z/p^n\Z$. Thus, one has the tower of Galois extensions
\[F=F_0 \subset F_1 \subset F_2 \subset \cdots \subset F_n \subset F_{n+1}\subset \cdots\subset F_{\op{cyc}}. \]
Let $e_n$ be the exact power of $p$ that divides the class number of $F_n$. Iwasawa proved that there exist invariants $\mu=\mu_p(F), \lambda=\lambda_p(F) \in \Z_{\geq 0}$ and $\nu=\nu_p(F)\in \Z$ such that \[e_n=p^n\mu+n\lambda+\nu,\] for all large enough values of $n$.
Moreover, Iwasawa conjectured that the $\mu_p(F)=0$ for all number fields $F$. The conjecture has been resolved for abelian extensions $F/\Q$ by Ferrero and Washington \cite{ferrerolawrence}.
\par Let $K$ be a number field and $L/K$ be a finite Galois extension such that $\op{Gal}(L/K)$ is a $p$-group. Kida \cite{Kid80} showed that there is an explicit relationship between the Iwasawa $\mu$- and $\lambda$-invariants for $K_{\op{cyc}}/K$ and those for $L_{\op{cyc}}/L$. This can be viewed as an analogue of the Riemann-Hurwitz formula for Iwasawa invariants. Iwasawa \cite[Theorem~6]{Iwa81} later proved a generalization of this result using Galois cohomology. Let $w$ be a prime of $L_{\op{cyc}}$ and $v$ be the prime of $K_{\op{cyc}}$ such that $w|v$. Set $e(w)$ to denote the ramification index of $w$ over $v$.
Let $U(L_{\op{cyc}})$ be the group of units of $L_{\op{cyc}}$.
\begin{lthm}[{\cite[Theorem~6]{Iwa81}}]
\label{thm: classical Kida's formula}
Let $p$ be a prime, $K$ be a number field and $L/K$ be a finite Galois extension such that $[L:K]$ is a power of $p$.
Assume that $\mu_p(K)=0$. Then, $\mu_p(L)=0$ and
\[
\lambda_p(L) = [L_{\op{cyc}}:K_{\op{cyc}}]\lambda_p(K) + \sum_{w}\left(e(w)-1\right) + (p-1)(h_2 -h_1).
\] 
The sum is over all primes $w$ of $L_{\op{cyc}}$ (above $v$ in $K_{\op{cyc}}$) not dividing $p$. The quantity $h_i$ is the rank of the abelian group $H^i(L_{\op{cyc}}/K_{\op{cyc}}, U(L_{\op{cyc}}))$.
\end{lthm}
\subsection{Main results}
\par Mazur \cite{Maz72} formulated the Iwasawa theory of elliptic curves with good ordinary reduction at the primes that lie above $p$. Kato \cite{Katomodforms} proved that for elliptic curves defined over $\Q$, the $p$-primary Selmer groups considered by Mazur were cofinitely generated and cotorsion over the Iwasawa algebra. Hachimori and Matsuno \cite{HachimoriMatsuno} proved a generalization of Kida's formula for Selmer groups of elliptic curves with good ordinary or multiplicative reduction at the primes that lie above $p$. The Iwasawa theory of elliptic curves with additive reduction at primes above $p$ was initially studied by Delbourgo \cite{Delbs}, who proved that the natural generalization of Kato's result should hold, provided additional assumptions are satisfied (see Proposition \ref{cotorsion and EC well defined propn}). It is natural therefore to extend Kida's formula to elliptic curves with additive reduction (at primes above $p$). Set $\Lambda$ to denote the Iwasawa algebra. Let $E$ be an elliptic curve over a number field $K$ and $L/K$ be a finite Galois extension with Galois group $G:=\op{Gal}(L/K)$. The Iwasawa $\mu$- and $\lambda$-invariants associated to the $p$-primary Selmer group $\op{Sel}_{p^\infty}(E/K_{\op{cyc}})$ are denoted by $\mu_p(E/K)$ and $\lambda_p(E/K)$ respectively. Stated below is our main result. 

\begin{lthm}\label{Kida formula main thm}
   With respect to notation above, assume that $G$ is a $p$-group. Moreover, assume that the following conditions are satisfied.
    \begin{enumerate}
        \item There exists a finite Galois extension $K'/K$ with Galois group $\Delta:=\op{Gal}(K'/K)$ over which $E$ has good reduction at the primes above $p$. Moreover, assume that $p\nmid |\Delta|$.
        \item Let $S_{\op{add}}$ be the set of primes $v$ of $K$ not dividing $p$ at which $E$ has additive reduction. Then all primes of $S_{\op{add}}$ continue to have additive reduction in $L_{\op{cyc}}$ (this condition is automatically satisfied when $L/K$ is unramified at all primes of $S_{\op{add}}$ or if $p\geq 5$). 
        \item The Selmer group $\op{Sel}_{p^\infty}(E/K_{\op{cyc}})$ is cofinitely generated and cotorsion over $\Lambda$ with $\mu_p(E/K)=0$.
    \end{enumerate}
    Then, the following assertions hold
    \begin{enumerate}
        \item The Selmer group $\op{Sel}_{p^\infty}(E/L_{\op{cyc}})$ is cofinitely generated and cotorsion over $\Lambda$ with $\mu_p(E/L)=0$. 
        \item We have that 
        \begin{equation}\label{main lambda eqn}\lambda_p(E/L)=[L_{\op{cyc}}:K_{\op{cyc}}] \lambda_p(E/K)+\sum_{w\in P_1} \left(e(w)-1\right)+2 \sum_{w\in P_2} \left(e(w)-1\right).\end{equation}
    \end{enumerate}
    Here, $P_1$ and $P_2$ are the sets of primes of $L_{\op{cyc}}$ defined as follows
    \[\begin{split}
       & P_1:=\{w\mid w\nmid p\text{, }E\text{ has split multiplicative reduction at }w\},\\
       & P_2:=\{w\mid w\nmid p\text{, }E\text{ has good reduction at }w\text{ and }E(L_{\op{cyc}, w})\text{ has a point of order }p\}. \\
    \end{split}\]
\end{lthm}

\par Note that the first of the above conditions imply that $E$ has potentially good reduction at the primes above $p$. Leveraging Delbourgo's results for $K=\Q$, we provide explicit conditions for the assumptions of Theorem \ref{Kida formula main thm} to hold, cf. Proposition \ref{cotorsion and EC well defined propn} and Corollary \ref{equivalent cond for Delbs ECF}.  We also give an explicit example to illustrate these results, see the example following Assumption \ref{last ass} on p. 20.
\par We then come to our main application, which is to prove density results for the stability of $\mu$ and $\lambda$-invariants in $\Z/p\Z$-extensions of $\Q$. Furthermore, one is also able to derive surprising results concerning rank stability in these extensions. There is considerable interest in rank stability questions for a fixed elliptic curve in prime cyclic extensions, see for instance \cite{DFK1, DFK2, mazurrubin, AnweshDebs, Anweshdioph}. 

\par Our results are proven via analytic methods, specifically by an application of Delange's Tauberian theorem (cf. Theorem \ref{tauberian thm}). Let $E_{/\Q}$ be an elliptic curve satisfying certain additional conditions (cf. Assumption \ref{last ass}). One of these conditions requires that the $p$-primary Selmer group of $E$ over $\Q_{\op{cyc}}$ is cotorsion over $\Lambda$, and the Iwasawa $\mu$- and $\lambda$-invariants are $0$. For instance, for $p=3$, the elliptic curve $E:y^2+y=x^3-3x-5$ is shown to satisfy these conditions. Then, we prove an asymptotic lower bound for the number of $\Z/p\Z$-extensions $L/\Q$ such that the $\mu$- and $\lambda$-invariants over for $L_{\op{cyc}}$ remain $0$. 

\par Given a number field $L$, set $\Delta_L$ to denote its discriminant. Let $X>0$, and $\mathcal{S}(X)$ be the set of Galois extensions $L/\Q$ with $\op{Gal}(L/\Q)\simeq \Z/p\Z$ and such that $|\Delta_L|\leq X$. Take $\mathcal{S}_E(X)$ to be subset of $\mathcal{S}(X)$ consisting of the extensions for which the following conditions hold
\begin{itemize}
    \item $\op{Sel}_{p^\infty}(E/L)$ is cofinitely generated and cotorsion over $\Lambda$,
    \item $\mu_p(E/L)=0$ and $\lambda_p(E/L)=0$.
\end{itemize}
 We note that for $L\in \mathcal{S}_E(X)$, it follows from Proposition \ref{propn rank<=lambda} that $\op{rank} E(L)=0$. Thus, the rank remains stable in such extensions $L/\Q$. We prove asymptotic formulae for $N_E(X):=\#\mathcal{S}_E(X)$ respectively. Given two positive real valued functions $f(X)$ and $g(X)$, we write $f(X)\gg g(X)$ to mean that there is a constant $C>0$ such that $C f(X)> g(X)$ for all sufficiently large $X$.
\begin{lthm}\label{thm C}
        Let $E_{/\Q}$ be an elliptic curve satisfying Assumption \ref{last ass}. With respect to notation above, assume that the Galois representation 
        \[\rho_{E, p}: \op{Gal}(\bar{\Q}/\Q)\rightarrow \op{GL}_2(\F_p)\] on the $p$-torsion in $E(\bar{\Q})$ is surjective. Then, we have that 
    \[N_E(X)\gg X^{\frac{1}{(p-1)}}(\log X)^{-\beta},\] where $\beta:=\frac{p^2-p+2}{p^3-p^2-p+1}$. 
\end{lthm}
On the other hand, if we let $M(X):=\# \mathcal{S}(X)$, then, there is a constant $c>0$ such that $M(X)\sim c X^{\frac{1}{p-1}}$. This is a special case of Malle's conjecture \cite{malle2002distribution, malle2004distribution}, and this particular result is due to M\"aki \cite{Maki}. Later, the count was generalized to arbitrary number field bases by Wright \cite{Wright}. The power of $\log X$ in Theorem \ref{thm C} is negative, however, still is very close to $0$ (especially for large values of $p$). On comparing $N_E(X)$ with $M(X)$, our result shows that there is a significantly large number of extensions in which the rank remains stable, compared to the total asymptotic. The difference simply lies in the power of $\log X$. We remark that for $p=3$ and $E:y^2+y=x^3-3x-5$, the conditions of Theorem \ref{thm C} are satisfied. 
\subsection{Organization} Including the introduction, the article consists of five sections. In Section \ref{s 2}, we discuss preliminary notions and set up notation. We begin by discussing the algebraic structure of Selmer groups considered over cyclotomic $\Z_p$-extensions. In this section, we also recall the results of Hachimori and Matsuno \cite{HachimoriMatsuno}. The Section \ref{s 3} is devoted to the proof of Theorem \ref{Kida formula main thm}. Following this, we discuss the notion of the Euler characteristic of a $\Lambda$-module in Section \ref{s 4}, and recall results of Delbourgo on elliptic curves over $\Q$ with additive reduction at $p$. These results are used in discussing precise conditions for the conditions of Theorem \ref{Kida formula main thm} are satisfied for $K=\Q$. Finally, in Section \ref{s 5}, Theorem \ref{thm C} is proven.

\section{Preliminaries}\label{s 2}
\par This section is preliminary in nature, and we set up notation that will be in place throughout the article. It serves to introduce basic notions in the Iwasawa theory of elliptic curves. For a more detailed exposition, the reader may refer to \cite{greenberg, Greenbergintro}.
\subsection{Selmer groups associated to elliptic curves}
\par Let $p$ be an odd prime number and $K$ be a fixed number field. Denote by $\Sigma_p$, the set of primes of $K$ that lie above $p$. Let $E$ be an elliptic curve that is defined over $K$. Denote by $\Sigma_{\op{bad}}$ the set of primes $v$ of $K$ at which $E$ has bad reduction. Let $\Sigma$ be a finite set of primes of $K$ that contains $\Sigma_{\op{bad}}$ and $\Sigma_p$. Throughout, we choose an algebraic closure $\bar{K}/K$ as well as $\bar{K}_v/K_v$ for any prime $v$ of $K$. At each prime $v$ of $K$, choose an embedding $\iota_v: \bar{K}\hookrightarrow \bar{K}_v$. We take $\op{G}_v$ to denote the absolute Galois group of $K_v$. Let $K_\Sigma$ be the maximal algebraic extension of $K$ in which all primes $v\notin \Sigma$ are unramified. Note that $K_\Sigma$ is a Galois extension of $K$. We set $\op{G}_{K, \Sigma}:=\op{Gal}(K_\Sigma/K)$. For each prime $v\in \Sigma$, we have a natural map 
\[\iota_v^*: \op{G}_v\rightarrow \op{G}_{K, \Sigma},\] which is induced by $\iota_v$. 
\par For $n\in \Z_{\geq 1}$, take $E[p^n]$ to be the $p^n$-torsion subgroup of $E(\bar{K})$. Set $E[p^\infty]$ to be the $p$-primary part of $E(\bar{K})$. The action of $\op{Gal}(\bar{K}/K)$ descends to an action of $\op{G}_{K,\Sigma}$ on $E[p^\infty]$ since $\Sigma$ contains the primes above $p$ and the primes at which $E$ has bad reduction. For $i\geq 0$ and any algebraic extension $\cK/K$ that is contained in $K_\Sigma$, we set
\[H^i(K_\Sigma/\cK, \cdot):=H^i(\op{Gal}(K_\Sigma/\cK), \cdot).\]
Let $\Z_p$ denote the ring of $p$-adic integers. There is a unique extension $K_{\op{cyc}}/K$ contained in $K(\mu_{p^\infty})$ such that $\op{Gal}(K_{\op{cyc}}/K)$ is topologically isomorphic to $\Z_p$. This extension is called the \emph{cyclotomic $\Z_p$-extension of $K$}. The only primes of $K$ that ramify in $K_{\op{cyc}}$ are those that lie above $p$, and hence $K_{\op{cyc}}$ is contained in $K_\Sigma$. We set $\Gamma$ to denote the Galois group $\op{Gal}(K_{\op{cyc}}/K)$ and choose a topological generator $\gamma$ of $\Gamma$. Given any integer $n\geq 0$, let $K_n/K$ be the extension contained in $K_{\op{cyc}}$ such that $[K_n:K]=p^n$. This extension is called the \emph{$n$-th layer}, and one has the following tower of extensions
\[K=K_0\subset K_1\subset K_2\subset \dots \subset K_n \subset K_{n+1}\subset \dots \subset K_{\op{cyc}}.\] Identify the Galois group $\op{Gal}(K_{\op{cyc}}/K_n)$ with $\Gamma^{p^n}$ and set $\Gamma_n:=\Gamma/\Gamma^{p^n}$. In particular, one may identify $\Gamma_n$ with $\op{Gal}(K_n/K)$. In the course of this section, we shall introduce certain Selmer groups considered over $K_{\op{cyc}}$. These will be considered as modules over a certain completed group algebra called the Iwasawa algebra $\Lambda$, which has nice properties. We take $\Lambda$ to denote the inverse limit
\[\Lambda:=\varprojlim_n \Z_p[\Gamma_n],\] with respect to natural quotient maps 
\[\pi_{m, n}: \Gamma_m\rightarrow \Gamma_n\] for $m\geq n\geq 0$. Letting $T$ denote $(\gamma-1)$ we identify $\Lambda$ with the formal power series ring $\Z_p\llbracket T\rrbracket$. 
\par Let $\cK$ be an algebraic extension of $K$ which is contained in $K_\Sigma$. We introduce the $p$-primary Selmer group of $E$ over $\cK$. Let $v$ be a prime of $\cK$, and consider the Kummer exact sequence of Galois modules 
\begin{equation}\label{ses}0\rightarrow E[p^n]\rightarrow E\xrightarrow{\times p^n} E\rightarrow 0.\end{equation} Associated to \eqref{ses}, we have the exact sequence in cohomology
\[0\rightarrow E(\cK_v)\otimes \Z/p^n \Z\xrightarrow{\delta_{v,n}} H^1(\cK_v, E[p^n])\xrightarrow{\phi_{v, n}} H^1(\cK_v, E)[p^n]\rightarrow 0.\]
Taking the direct limit as $n\rightarrow \infty$, one obtains the following exact sequence
\[0\rightarrow E(\cK_v)\otimes \Q_p/\Z_p \xrightarrow{\delta_{v}} H^1(\cK_v, E[p^\infty])\xrightarrow{\phi_{v}} H^1(\cK_v, E)[p^\infty]\rightarrow 0.\]
\begin{definition}
    With respect to the above notation, the $p$-primary Selmer group of $E$ over $\cK$ is defined as follows
\[\op{Sel}_{p^\infty}(E/\cK):=\op{ker}\left( H^1(K_\Sigma/\cK, E[p^\infty])\xrightarrow{\Phi_{E, \cK}} \prod_{v} H^1(\cK_v, E)[p^\infty]\right),\] where the map $\Phi_{E, \cK}$ is the product of maps $\phi_v$ as $v$ ranges over primes of $\cK$ that lie above $\Sigma$.
\end{definition} Of particular interest is the Selmer group $\op{Sel}_{p^\infty}(E/K_{\op{cyc}})$, which may in fact be identified with the direct limit $\varinjlim_n \op{Sel}_{p^\infty}(E/K_n)$. As is well known, $\op{Sel}_{p^\infty}(E/K_{\op{cyc}})$ has a canonical $\Lambda$-module structure. Let $\Sigma_{\op{cyc}}$ be the set of primes of $K_{\op{cyc}}$ that lie above $\Sigma$. Note that all primes of $K$ are finitely decomposed in $K_{\op{cyc}}$, hence $\Sigma_{\op{cyc}}$ is finite. Thus, the Selmer group $\op{Sel}_{p^\infty}(E/K_{\op{cyc}})$ is the kernel of the natural restriction map $$H^1(K_\Sigma/K_{\op{cyc}}, E[p^\infty])\xrightarrow{\Phi_{E, K_{\op{cyc}}}} \prod_{v \in \Sigma_{\op{cyc}}} H^1(K_{\op{cyc}, v}, E)[p^\infty].$$ 
\subsection{Iwasawa invariants} In this subsection, we introduce the algebraic Iwasawa invariants associated to the $p$-primary Selmer group $\op{Sel}_{p^\infty}(E/K_{\op{cyc}})$. A polynomial $f(T)$ with coefficients in $\Z_p$ is said to be \emph{distinguished} if it is monic and all non-leading coefficients are divisible by $p$. Let $M$ and $M'$ be finitely generated and torsion $\Lambda$-modules. A $\Lambda$-module homomorphism $\phi: M\rightarrow M'$ is said to be a \emph{pseudo-isomorphism} if the kernel and cokernel of $\phi$ are finite. It follows from the structure theory of $\Lambda$-modules (cf. \cite[Chapter 13]{washingtoncyclotomicfields}) that if $M$ is a finitely generated and torsion $\Lambda$-module, then it is pseudo-isomorphic to a module $M'$ of the following form  \begin{equation}\label{structure decomposition M'}M'=\left(\bigoplus_{i=1}^s \Lambda/(p^{m_i})\right)\oplus \left(\bigoplus_{j=1}^t \Lambda/(f_j(T)^{n_j})\right).\end{equation} In \eqref{structure decomposition M'}, $s, t$ are non-negative integers, $m_i, n_j$ are positive integers, and $f_j(T)$ are irreducible distinguished polynomials. It is understood in the above notation that if $s$ (resp. $t$) is $0$, then the direct sum is empty. \begin{definition}
    The $\mu$-invariant is given by 
\[\mu_p(M):=\begin{cases} \sum_i m_i & \text{ if }s>0;\\
0 & \text{ if }s=0.
\end{cases}\]
On the other hand, the $\lambda$-invariant is given by 
\[\lambda_p(M):=\begin{cases} \sum_j n_j\op{deg}(f_j) & \text{ if }t>0;\\
0 & \text{ if }t=0.
\end{cases}\]
\end{definition}

 \begin{lemma}\label{M torsion mu=0 implies finitely generated Z_p module}
      Let $M$ be finitely generated $\Lambda$-module. Then, the following conditions are equivalent
      \begin{enumerate}
          \item $M$ is torsion with $\mu_p(M)=0$, 
          \item $M$ is finitely generated as a $\Z_p$-module. 
      \end{enumerate}
      Moreover, if these equivalent conditions are satisfied, then 
      \[\lambda_p(M)=\op{rank}_{\Z_p}(M).\]
  \end{lemma}
  \begin{proof}
Let us assume that $M$ is a finitely generated torsion $\Lambda$-module with $\mu(M)=0$. Then $M$ is pseudo-isomorphic to a $\Lambda$-module $M'$ given as follows
 \[M'= \left(\bigoplus_{j=1}^t \Lambda/(f_j(T)^{n_j})\right),\]
 with $f_j(T)$ distinguished polynomials. Then $\Lambda/(f_j(T)^{n_j})$ is finitely generated as a $\Z_p$-module with rank $n_j \op{deg} f_j$. Therefore, $M'$ is a finitely generated $\Z_p$-module and $\op{rank}_{\Z_p}(M') = \Sigma_{j} n_j \op{deg} f_j = \lambda_p(M')$. Since $M$ is pseudo-isomorphic to $M'$ as $\Lambda$-module, it follows that $\op{rank}_{\Z_p}(M) = \op{rank}_{\Z_p}(M')$ and $\lambda_p(M)=\lambda_p(M')$. 

 Conversely, assume that $M$ is a finitely generated $\Z_p$-module. It is easy to see that $M$ must be a torsion $\Lambda$-module. We show that $\mu_p(M)=0$. Suppose by the way of contradiction that $\mu_p(M) \neq 0$. Then from the structure theorem $M$ is pseudo-isomorphic to a $\Lambda$-module $M'$, which  contains $\Lambda/(p^m)$ for some $m >0$. Note that $\Lambda/(p^m)$ can be identified with $\Z/(p^m)[T]$ and therefore is infintely generated as a $\Z_p$-module. This gives a contradiction and hence $\mu_p(M)=0$.
  \end{proof}

Let $\mathfrak{X}(E/K_{\op{cyc}})$ denote the Pontryagin dual of $\op{Sel}_{p^\infty}(E/K_{\op{cyc}})$. Throughout the article, we make the following assumption. 
\begin{ass}\label{torsion hyp}
    The dual Selmer group $\mathfrak{X}(E/K_{\op{cyc}})$ is finitely generated and torsion as a $\Lambda$-module.
\end{ass}

When $E$ has good ordinary reduction at all primes above $p$, it is conjectured by Mazur that Assumption \ref{torsion hyp} holds. This conjecture was settled by Kato and Rubin in the case when $E$ is defined over $\Q$ and $K/\Q$ is an abelian extension. When $E$ has multiplicative reduction at the primes above $p$, it is still conjectured that Assumption \ref{torsion hyp} holds \cite[Introduction]{HachimoriMatsuno}. On the other hand, for $K=\Q$ and $E$ an elliptic curve with bad additive reduction at $p$, the conjecture was proven by Delbourgo \cite{Delbs} under additional hypotheses. 

\begin{definition}\label{def of Iwasawa invariants of Selmer groups}
    When the above assumption holds, we shall denote by $\mu_p(E/K)$ (resp. $\lambda_p(E/K)$) the $\mu$-invariant (resp. $\lambda$-invariant) of $\mathfrak{X}(E/K_{\op{cyc}})$. 
\end{definition}

\subsection{The results of Hachimori and Matsuno} In this subsection, we recall the results of Hachimori and Matsuno \cite{HachimoriMatsuno}, who prove an analogue of Kida's formula for elliptic curves $E_{/K}$ with semistable reduction at the primes of $\Sigma_p$. Let $L/K$ be a Galois extension of number fields with Galois group $G=\op{Gal}(L/K)$. We assume that $\# G$ is a power of $p$. Let $\Sigma_{\op{add}}$ be the set of primes of $K$ at which $E$ has additive reduction. For a prime $w$ of $L_{\op{cyc}}$, denote by $e(w)=e_{L_{\op{cyc}}/K_{\op{cyc}}}(w)$ the ramification index of $w$ over $K_{\op{cyc}}$. We recall their main result in the case when $E$ has good ordinary reduction at the primes in $\Sigma_p$.

\begin{theorem}[Hachimori-Matsuno]\label{HMthm}
    With respect to notation above, assume that the following conditions hold
    \begin{enumerate}
        \item $E$ has good ordinary reduction at all primes of $K$ that lie above $p$. 
        \item All primes of $\Sigma_{\op{add}}$ continue to have additive reduction in $L_{\op{cyc}}$ (this condition is automatically satisfied when $L/K$ is unramified at all primes of $\Sigma_{\op{add}}$ or if $p\geq 5$). 
        \item The Selmer group $\op{Sel}_{p^\infty}(E/K_{\op{cyc}})$ is cofinitely generated and cotorsion over $\Lambda$ with $\mu_p(E/K)=0$.
    \end{enumerate}
    Then, the following assertions hold
    \begin{enumerate}
        \item The Selmer group $\op{Sel}_{p^\infty}(E/L_{\op{cyc}})$ is cofinitely generated and cotorsion over $\Lambda$ with $\mu_p(E/L)=0$. 
        \item We have that 
        \[\lambda_p(E/L)=[L_{\op{cyc}}:K_{\op{cyc}}] \lambda_p(E/K)+\sum_{w\in P_1} \left(e(w)-1\right)+2 \sum_{w\in P_2} \left(e(w)-1\right).\]
    \end{enumerate}
    Here, $P_1$ and $P_2$ are the sets of primes of $L_{\op{cyc}}$ defined as follows
    \[\begin{split}
       & P_1:=\{w\mid w\nmid p\text{, }E\text{ has split multiplicative reduction at }w\},\\
       & P_2:=\{w\mid w\nmid p\text{, }E\text{ has good reduction at }w\text{ and }E(L_{\op{cyc}, w})\text{ has a point of order }p\}. \\
    \end{split}\]
\end{theorem}
\begin{proof}
    The above result is \cite[Theorem 3.1]{HachimoriMatsuno}.
\end{proof}
Hachimori and Matsuno prove a similar result in the case when $E$ has split multiplicative reduction at all primes in $\Sigma_p$, cf. \cite[section 8]{HachimoriMatsuno}.
\section{Growth of Iwasawa Invariants}\label{s 3}
Throughout the section, let $E$ be an elliptic curve over $K$ and let $L/K$ be a finite Galois extension with Galois group $G\simeq \Z/p\Z$. We shall eventually reduce our proof for an arbitrary finite $p$-group $G$ to this case. We make the Assumption \ref{torsion hyp} which states that the Selmer group  
$\op{Sel}_{p^\infty}(E/K_{\op{cyc}})$ is co-torsion as a 
$\Lambda$-module. Next, we make the following assumption on the reduction type of $E$ at the primes above $p$.
\begin{ass}\label{reduction type}
There exists a finite Galois extension $K'/K$ such that $E$ has good reduction at all primes $v \mid p$ of $K'$. Furthermore, setting  $\Delta:=\op{Gal}(K'/K)$, assume that $p \nmid \#\Delta$. 
\end{ass}

The case we shall be mostly interested in is when $E$ has additive reduction at some of the primes in $\Sigma_p$. We shall give an example when the Assumption \ref{torsion hyp} holds for $K=\Q$, $E$ has good ordinary reduction at the primes $v|p$ of $K'$ and $K'\subseteq \Q(\mu_p)$. This leverages results of Delbourgo \cite{Delbs} which we discuss in detail in the next section.

\par The following result of Imai \cite{Imai} will be of much significance in our local computations. 

\begin{lemma}\label{finite p-primary torsion}
    When Assumption \ref{reduction type} holds, we have that $E(K_{\op{cyc}})[p^\infty]$ is finite.
\end{lemma}

\begin{proof}
    Since the Assumption \ref{reduction type} holds, $E$ has good reduction at the primes of $K'$ that lie above $p$. Thus, $E(K'_{\op{cyc}})[p^\infty]$ is finite by the main theorem of \cite{Imai}. It follows that $E(K_{\op{cyc}})[p^\infty]$ is finite as well.
\end{proof}

\begin{proposition}\label{propn rank<=lambda}
    Let $E$ be an elliptic curve over a number field $K$ such that the Assumptions \ref{torsion hyp} and \ref{reduction type} hold. Moreover, assume that $\mu_p(E/K)=0$. Then, we have that $\op{rank} E(K)\leq \lambda_p(E/K)$.
\end{proposition}
\begin{proof}
    Note that \begin{equation}\label{comparison eq 1}\op{rank} E(K)\leq \op{corank}_{\Z_p} \left(\op{Sel}_{p^\infty}(E/K)\right),
    \end{equation}
    with equality if and only if $\Sh(E/K)[p^\infty]$ is finite. On the other hand, it follows from Lemma \ref{M torsion mu=0 implies finitely generated Z_p module} that $\op{Sel}_{p^\infty}(E/K_{\op{cyc}})$ is cofinitely generated as a $\Z_p$-module, and
    \begin{equation}\label{comparison eq 2}\lambda_p(E/K)=\op{corank}_{\Z_p} \left(\op{Sel}_{p^\infty}(E/K_{\op{cyc}})\right).\end{equation}
    Therefore, from \eqref{comparison eq 1} and \eqref{comparison eq 2}, it suffices to show that 
    \begin{equation}\label{comparison eq 3}
    \op{corank}_{\Z_p} \left(\op{Sel}_{p^\infty}(E/K)\right)\leq \op{corank}_{\Z_p} \left(\op{Sel}_{p^\infty}(E/K_{\op{cyc}})\right).
    \end{equation}
    Note that there is a comparison map
    \[\psi: \op{Sel}_{p^\infty}(E/K)\rightarrow \op{Sel}_{p^\infty}(E/K_{\op{cyc}})\] that is induced by the restriction map
    \[\op{res}: H^1(K, E[p^\infty])\rightarrow H^1(K_{\op{cyc}}, E[p^\infty]).\]
From the inflation-restriction sequence, the kernel of $\op{res}$ can be identified with $H^1(\Gamma, E(K_{\op{cyc}})[p^\infty])$. It follows from Lemma \ref{finite p-primary torsion} that $E(K_{\op{cyc}})[p^\infty]$ is finite, and hence, the kernel of $\psi$ is finite. This proves that the relation \eqref{comparison eq 3} holds.\\
Thus the result has been proved.
\end{proof}

\begin{proposition}\label{surjectivity propn}
    When Assumption \ref{torsion hyp} holds, then the restriction map $$\Phi_{E, K_{\op{cyc}}}: H^1(K_\Sigma/K_{\op{cyc}}, E[p^\infty]) \to \prod_{v \in \Sigma_{\op{cyc}}} H^1(K_{\op{cyc}, v}, E)[p^\infty]$$ is surjective.
\end{proposition}
\begin{proof}
    The result follows from Lemma \ref{finite p-primary torsion} and the proof of \cite[Lemma~3.4]{AnweshDebs} verbatim.
\end{proof}
\begin{lemma}\label{lemma 3.4}
    Let $v$ be a prime of $K_{\op{cyc}}$ lying above $p$ and $w$ be a prime of $L_{\op{cyc}}$ lying above $v$. Suppose that Assumption \ref{reduction type} holds, then \[H^i(G,E(L_{\op{cyc},w}))=0\text{ for }i=1,2.\]
\end{lemma}
\begin{proof}
    Recall from Assumption \ref{reduction type} that there exists a finite Galois extension $K'/K$ such that $E$ has good reduction at all primes of $K'$ dividing $p$. Furthermore, setting $\Delta:=\op{Gal}(K'/K)$, we assume that $p \nmid |\Delta|$. Setting $L':=L \cdot K'$, we let $w'$ be a prime of $L'_{\op{cyc}}$ that lies above $w$. Take $v'$ to be the prime of $K'_{\op{cyc}}$ that lies below $w'$, as depicted below
   
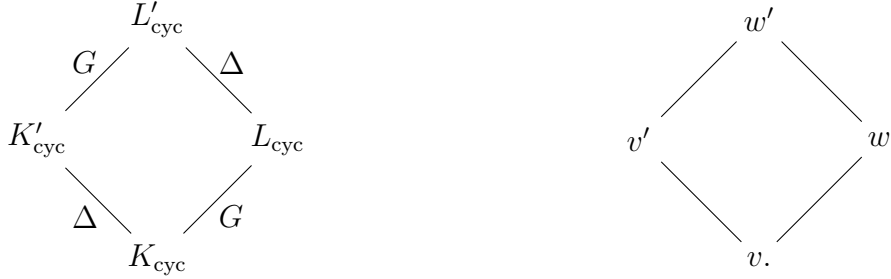
\[ \begin{tikzpicture}[scale=.8]
    \begin{scope}[xshift=0cm]
    \node (Q1) at (0,0) {$K_{\op{cyc}}$};
    \node (Q2) at (2,2) {$L_{\op{cyc}}$};
    \node (Q3) at (0,4) {$L'_{\op{cyc}}$};
    \node (Q4) at (-2,2) {$K'_{\op{cyc}}$};

    \draw (Q1)--(Q2) node [pos=0.7, below,inner sep=0.25cm] {$G$};
    \draw (Q1)--(Q4) node [pos=0.7, below,inner sep=0.25cm] {$\Delta$};
    \draw (Q3)--(Q4) node [pos=0.7, above,inner sep=0.25cm] {$G$};
    \draw (Q2)--(Q3) node [pos=0.3, above,inner sep=0.25cm] {$\Delta$};
    \end{scope}

    \begin{scope}[xshift=10cm]
    \node (Q1) at (0,0) {$v.$};
    \node (Q2) at (2,2) {$w$};
    \node (Q3) at (0,4) {$w'$};
    \node (Q4) at (-2,2) {$v'$};

    \draw (Q1)--(Q2) node [pos=0.7, below,inner sep=0.25cm] {};
    \draw (Q1)--(Q4) node [pos=0.7, below,inner sep=0.25cm]{};
    \draw (Q3)--(Q4) node [pos=0.7, above,inner sep=0.25cm]{};
    \draw (Q2)--(Q3) node [pos=0.7, above,inner sep=0.25cm]{};
    \end{scope}
    \end{tikzpicture}
\]
Since $\Delta$ has order prime to $p$ and $G\simeq \Z/p\Z$ by assumption, it follows that $L_{\op{cyc}}\cap K'_{\op{cyc}}=K_{\op{cyc}}$ and $\op{Gal} \left(K'_{\op{cyc}}/K_{\op{cyc}} \right)$ is naturally isomorphic to $\Delta$. Likewise, $\op{Gal} \left(L'_{\op{cyc}}/K'_{\op{cyc}} \right)$ can be identified with $G$. 
\par Let us first prove the result for $i=1$, i.e., $H^1(G, E(L_{\op{cyc}, w}))=0$.
Since $E$ has good reduction at the primes of $L'_{\op{cyc}}$ that lie above $p$, it follows from \cite[p. 592, l.3, proof of Lemma~4.3]{HachimoriMatsuno} that \begin{equation}\label{H1 G vanishing} H^1(G,E(L'_{\op{cyc},w'}))=0.\end{equation}
We note that this makes use of a deep result of Coates and Greenberg \cite[Theorem 3.1]{CoatesGreenberg}.
    From the inflation-restriction sequence applied to $L'_{\op{cyc}}/K'_{\op{cyc}}/ K_{\op{cyc}}$, we get
    \begin{equation}\label{infres1}
    0 \rightarrow H^1(\Delta,E(K'_{\op{cyc},v'})) \xrightarrow{\op{inf}} H^1(L'_{\op{cyc}}/K_{\op{cyc}}, E(L'_{\op{cyc},w'})) \xrightarrow{\op{res}} H^1(G, E(L'_{\op{cyc},w'}))^\Delta.\end{equation}
    Note that from \eqref{H1 G vanishing} we have that $H^1(G, E(L'_{\op{cyc},w'}))^\Delta=0$. On the other hand, it follows from restriction-corestriction that every element of $H^1(\Delta,E(K'_{\op{cyc},v'}))$ has order dividing $|\Delta|$. In particular, $H^1(\Delta,E(K'_{\op{cyc},v'}))[p^\infty]=0$ and thus, from \eqref{infres1} we find that 
    \begin{equation}\label{H1 L'/K=0}   H^1(L'_{\op{cyc}}/K_{\op{cyc}}, E(L'_{\op{cyc},w'}))[p^\infty]=0.
    \end{equation}
On the other hand, the inflation-restriction sequence applied to $L'_{\op{cyc}}/L_{\op{cyc}}/ K_{\op{cyc}}$ gives us
    \[0 \rightarrow H^1(G,E(L_{\op{cyc},w})) \xrightarrow{\op{inf}} H^1(L'_{\op{cyc}}/K_{\op{cyc}}, E(L'_{\op{cyc},w'})) \xrightarrow{\op{res}} H^1(\Delta, E(L'_{\op{cyc},w'}))^G.\]
   From \eqref{H1 L'/K=0} and the injectivity of the inflation map, we obtain that \[H^1(G,E(L_{\op{cyc},w}))[p^\infty]=0.\] Since $G\simeq \Z/p\Z$, it follows (once again, from the restriction-corestriction sequence) that multiplication by $p$ is equal to $0$ on $H^1(G, E(L_{\op{cyc}, w}))$. Hence, $H^1(G,E(L_{\op{cyc},w}))$  is a $p$-group, and we conclude that 
   \[H^1(G,E(L_{\op{cyc},w}))=0.\]
Next, we prove that $H^2(G,E(L_{\op{cyc},w}))=0$ via a similar argument. It is stated on p.592, l.3 of the proof of \cite[Lemma 4.3]{HachimoriMatsuno} that $H^i\left(G, E(L'_{\op{cyc}, w'})\right)=0$ for $i=1,2$. Since $G$ is cyclic, it is thus true that $H^i\left(G, E(L'_{\op{cyc}, w'})\right)=0$ for all $i>0$. \par From \cite[Corollary 2.4.2]{NSW}, we deduce that 
\[H^2\left(L'_{\op{cyc}, w'}/K_{\op{cyc}, v}, E(L'_{\op{cyc}, w'})\right)\simeq H^2(\Delta, E(K'_{\op{cyc}, v'})).\] Since $p\nmid |\Delta|$, we have that 
\begin{equation}\label{H^2 vanishing boring equation}H^2\left(L'_{\op{cyc}, w'}/K_{\op{cyc}, v}, E(L'_{\op{cyc}, w'})\right)[p^\infty]=0.\end{equation}
As was mentioned previously in our proof, 
\[H^1(\Delta, E(L'_{\op{cyc},w'}))[p^\infty]=0.\]
From the inflation-restriction sequence
$$H^1(\Delta, E(L'_{\op{cyc},w'}))^G \rightarrow H^2(G,E(L_{\op{cyc},w})) \rightarrow H^2(L'_{\op{cyc}}/K_{\op{cyc}}, E(L'_{\op{cyc},w'})),$$
and \eqref{H^2 vanishing boring equation} we deduce that $H^2(G,E(L_{\op{cyc},w}))[p^\infty]=0$. Since  $H^2(G,E(L_{\op{cyc},w}))$ is a $p$-group, we conclude that $H^2(G,E(L_{\op{cyc},w}))=0$ which completes the proof for $i=2$.
\end{proof}
 Let $\Sigma(K_{\op{cyc}})$ (resp. $\Sigma(L_{\op{cyc}})$) be the set of nonarchimedian primes of $K_{\op{cyc}}$ (resp. $L_{\op{cyc}}$) that lie above $\Sigma$. 

  \begin{lemma}\label{alpha finite kernel and cokernel}
      There is a natural map 
      \[\alpha: \op{Sel}_{p^\infty}(E/K_{\op{cyc}})\rightarrow \op{Sel}_{p^\infty}(E/L_{\op{cyc}})^G\] whose kernel and cokernel are both finite.
  \end{lemma}
  \begin{proof}
      The proof follows from the argument of \cite[Lemma 3.3]{HachimoriMatsuno} and uses the assertion of Lemma \ref{lemma 3.4}. 
  \end{proof}

\begin{proposition}\label{boring prop 1}
    Assume that $\mathfrak{X}(E/K_{\op{cyc}})$ is a torsion $\Lambda$-module with $\mu_p(E/K)=0$. Then it follows that $\mathfrak{X}(E/L_{\op{cyc}})$ is a torsion $\Lambda$-module with $\mu_p(E/L)=0$. Moreover, the $\lambda$-invariants are given as follows
    \begin{equation}\label{lambda invariants}\begin{split}
&\lambda_p(E/K)=\op{corank}_{\Z_p}\op{Sel}_{p^\infty}(E/L_{\op{cyc}})^G,\\
&\lambda_p(E/L)=\op{corank}_{\Z_p}\op{Sel}_{p^\infty}(E/L_{\op{cyc}}).
    \end{split}\end{equation}
\end{proposition}
\begin{proof}
    The proof of this result is identical to that of \cite[Corollary 3.4]{HachimoriMatsuno}. Nevertheless we provide some details for the benefit of exposition. Since $\mathfrak{X}(E/K_{\op{cyc}})$ is assumed to be torsion as a $\Lambda$-module whose $\mu$-invariant vanishes, it follows from the Lemma \ref{M torsion mu=0 implies finitely generated Z_p module} that $\mathfrak{X}(E/K_{\op{cyc}})$ is finitely generated as a $\Z_p$-module. The Lemma \ref{alpha finite kernel and cokernel} asserts that there is a map 
    \[\alpha: \op{Sel}_{p^\infty}(E/K_{\op{cyc}})\rightarrow \op{Sel}_{p^\infty}(E/L_{\op{cyc}})^G\] with finite kernel and cokernel. As a result, we deduce that $\mathfrak{X}(E/L_{\op{cyc}})_G$ is a finitely generated $\Z_p$-module. From Nakayama's lemma applied to compact modules over local rings (cf. \emph{loc. cit.} for a precise reference) it follows that $\mathfrak{X}(E/L_{\op{cyc}})$ is finitely generated as a $\Z_p[G]$-module. Since $G$ is finite, we deduce that $\mathfrak{X}(E/L_{\op{cyc}})$ is finitely generated as a $\Z_p$-module. By Lemma \ref{M torsion mu=0 implies finitely generated Z_p module} we conclude that $\mathfrak{X}(E/L_{\op{cyc}})$ is torsion over $\Lambda$ and $\mu_p(E/L)=0$. 
    \par Lemma \ref{M torsion mu=0 implies finitely generated Z_p module} asserts that 
    \[\begin{split}
&\lambda_p(E/K)=\op{corank}_{\Z_p}\op{Sel}_{p^\infty}(E/K_{\op{cyc}}),\\
&\lambda_p(E/L)=\op{corank}_{\Z_p}\op{Sel}_{p^\infty}(E/L_{\op{cyc}}).
    \end{split}\]
    It follows from Lemma \ref{alpha finite kernel and cokernel} that
    \[\op{corank}_{\Z_p}\op{Sel}_{p^\infty}(E/K_{\op{cyc}})=\op{corank}_{\Z_p}\op{Sel}_{p^\infty}(E/L_{\op{cyc}})^G.\] Thus, we have proven the assertions regarding the $\lambda$-invariants \eqref{lambda invariants}.
\end{proof}

\par Let $A$ be a cofinitely generated $\Z_p[G]$-module. Recall that the \emph{Herbrand quotient} is defined as follows:
\[h_G(A):=\frac{\# H^2(G, A)}{\# H^1(G, A)}.\]

\begin{proposition}\label{relation between lambda invariants}
Assume that $\mathfrak{X}(E/K_{\op{cyc}})$ is torsion as a $\Lambda$-module with $\mu_p(E/K)=0$, then 
\[\lambda_p(E/L)=p \lambda_p(E/K)+(p-1)h_G\left( \op{Sel}_{p^\infty}(E/L_{\op{cyc}})\right).\]
\end{proposition}
\begin{proof}
    The above result follows from the arguments on \cite[p. 589, (3.3)]{HachimoriMatsuno}.
\end{proof}

In order to complete the proof of the result, we only need to compute the Herbrand quotient $h_G\left(\op{Sel}_{p^\infty}(E/L_{\op{cyc}})\right)$. Recall that it is assumed that $\op{Sel}_{p^\infty}(E/K_{\op{cyc}})$ is cotorsion as a $\Lambda$-module and $\mu_p(E/K)=0$. It follows therefore from Proposition \ref{boring prop 1} that the same is true for $\op{Sel}_{p^\infty}(E/L_{\op{cyc}})$. Thus from Proposition \ref{surjectivity propn} that there is a short exact sequence
\[0\rightarrow \op{Sel}_{p^\infty}(E/L_{\op{cyc}})\rightarrow H^1(K_\Sigma/L_{\op{cyc}}, E[p^\infty])\rightarrow \prod_{w\in \Sigma(L_{\op{cyc}})} H^1(L_{\op{cyc}, w}, E)[p^\infty]\rightarrow 0.\]
Therefore, we find that 
\begin{equation}\label{herbrand equation}
    h_G\left(\op{Sel}_{p^\infty}(E/L_{\op{cyc}})\right)=\frac{h_G\left(H^1(K_\Sigma/L_{\op{cyc}}, E[p^\infty])\right)}{\prod_{v\in \Sigma(K_{\op{cyc}})}h_G\left( \prod_{w|v} H^1(L_{\op{cyc}, w}, E)[p^\infty]\right)}.
\end{equation}

We first compute the Herbrand quotient $h_G\left(H^1(K_\Sigma/L_{\op{cyc}}, E[p^\infty])\right)$ following arguments in \cite{HachimoriMatsuno}. 

\begin{lemma}
    We have that 
    \[h_G\left(H^1(K_\Sigma/L_{\op{cyc}}, E[p^\infty])\right)=h_G\left(E(L_{\op{cyc}})[p^\infty]\right)=1.\]
\end{lemma}
\begin{proof}
    The first equality 
    \[h_G\left(H^1(K_\Sigma/L_{\op{cyc}}, E[p^\infty])\right)=h_G\left(E(L_{\op{cyc}})[p^\infty]\right)\] follows verbatim from \cite[Lemma 4.1]{HachimoriMatsuno}. Since $E$ has good reduction at all primes of $L'$ that lie above $p$, it follows from a result of Imai \cite{Imai} that $E(L'_{\op{cyc}})[p^\infty]$ is finite. Consequently, we deduce that $E(L_{\op{cyc}})[p^\infty]$ is finite and 
    \[h_G\left(E(L_{\op{cyc}})[p^\infty]\right)=1.\]
\end{proof}

Next, we compute the local Herbrand quotients 
\[h_{G,v}:=h_G\left( \prod_{w|v} H^1(L_{\op{cyc}, w}, E)[p^\infty]\right).\] The computation for $h_{G,v}$ for primes $v\nmid p$ follows verbatim from those done in \cite{HachimoriMatsuno}, cf. \cite[Lemma 4.2, Proposition 5.1, Corollary 5.2]{HachimoriMatsuno}. It remains for us to compute $h_{G,v}$ for primes $v|p$. 
\begin{lemma}
    For a prime $v|p$, we have that $h_{G,v}=1$.
\end{lemma}
\begin{proof}
    The result follows directly from Lemma \ref{lemma 3.4} which asserts that $H^i(G,E(L_{\op{cyc},w}))=0$ for $i=1,2$.
\end{proof}

\begin{lemma}\label{reduction to cyclic case lemma}
    Let $K \subset L \subset M$ be number fields such that $L/K$ and $M/K$ are Galois $p$-extensions. If the assertion of Theorem \ref{Kida formula main thm} holds for $M/L$ and $L/K$, then it holds for $M/K$.
\end{lemma}
\begin{proof}
    The proof of the result follows verbatim from the case considered by Hachimori and Matsuno (cf. \cite[Lemma 3.2]{HachimoriMatsuno}).
\end{proof}
We now give the proof of the main result. 
\begin{proof}[Proof of Theorem \ref{Kida formula main thm}]
    Let $G:=\op{Gal}(L_{\op{cyc}}/K_{\op{cyc}})$. It then follows from Lemma \ref{reduction to cyclic case lemma} that we may assume without loss of generality that $G\simeq \Z/p\Z$. Since it is assumed that $\mathfrak{X}(E/K_{\op{cyc}})$ is torsion as a $\Lambda$-module with $\mu_p(E/K)=0$, it follows from Proposition \ref{boring prop 1} that the same assertion is true for $\mathfrak{X}(E/L_{\op{cyc}})$. Then, recall that from Proposition \ref{relation between lambda invariants} that 
    \[\lambda_p(E/L)=p \lambda_p(E/K)+(p-1)h_G\left( \op{Sel}_{p^\infty}(E/L_{\op{cyc}})\right).\] It remains to compute the Herbrand quotient $h_G\left( \op{Sel}_{p^\infty}(E/L_{\op{cyc}})\right)$. Recall that from \eqref{herbrand equation} we have that 
    \[ h_G\left(\op{Sel}_{p^\infty}(E/L_{\op{cyc}})\right)=\frac{h_G\left(H^1(K_\Sigma/L_{\op{cyc}}, E[p^\infty])\right)}{\prod_{v\in \Sigma(K_{\op{cyc}})}h_{G,v}},\] where 
    \[h_{G,v}:=h_G\left( \prod_{w|v} H^1(L_{\op{cyc}, w}, E)[p^\infty]\right).\]
    Let $v|p$ be a prime of $K_{\op{cyc}}$, then from Lemma \ref{lemma 3.4} we find that $h_{G,v}=1$. For $v\nmid p$, the computation of $h_{G,v}$ follows from \cite[Lemma 4.2, Proposition 5.1, Corollary 5.2]{HachimoriMatsuno}. Putting it all together, we obtain the formula relating $\lambda$-invariants
    \[\lambda_p(E/L)=[L:K] \lambda_p(E/K)+\sum_{w\in P_1} \left(e(w)-1\right)+2 \sum_{w\in P_2} \left(e(w)-1\right).\] 
\end{proof}

 We recall that $P_1$ and $P_2$ are the sets of primes of $L_{\op{cyc}}$ defined as follows
    \[\begin{split}
       & P_1:=\{w\mid w\nmid p\text{, }E\text{ has split multiplicative reduction at }w\},\\
       & P_2:=\{w\mid w\nmid p\text{, }E\text{ has good reduction at }w\text{ and }E(L_{\op{cyc}, w})\text{ has a point of order }p\}. \\
    \end{split}\]

    Let us discuss conditions under which
    \[\sum_{w\in P_i} \left(e(w)-1\right)=0\] for $i=1,2$. We note here that the sums $\sum_{w\in P_i} (e(w)-1)$ are supported only at primes $w\nmid p$ of $L_{\op{cyc}}$ that are ramified over $K_{\op{cyc}}$.

    The following well known fact will prove to be very useful in our calculations.
\begin{lemma}\label{NSW lemma}
Let $G$ be a finite abelian group of $p$-power order and $M$ be a $p$-primary $G$-module. Suppose that $M^G=0$ or $M_G=0$, then $M=0$.
\end{lemma}
\begin{proof}
The stated result is \cite[Proposition 1.6.12]{NSW}.
\end{proof}

    \begin{definition}\label{defn of the Q_is}
    We introduce some further notation.
    \begin{itemize}
        \item Let $Q_1$ be the set of primes $v\nmid p$ of $K$ at which $E$ has bad reduction. Note that the set of primes $Q_1$ is finite.
        \item Let $Q_2$ be the set of all primes $v\nmid p$ of $K$ at which $E$ has good reduction and $p$ divides $\# \widetilde{E}(k_v)$. Here, $k_v$ is the residue field of $v$ and $\widetilde{E}$ is the reduction of $E$ at $v$. This set of primes is possibly infinite.
        \item Let $Q_3$ consist of the primes $v\nmid p$ of $K$ that are in the complement of $Q_1\cup Q_2$.
    \end{itemize}
\end{definition}

\begin{lemma}\label{w in P_2 and Q_2}
    Let $w$ be a prime of $L_{\op{cyc}}$ and $v$ be the prime of $K$ such that $w|v$. Assume that $E$ has good reduction at $v$ and that $w\in P_2$. Then, we find that $v\in Q_2$.
\end{lemma}
\begin{proof}
    \par Suppose that $w\in P_2$. Then, there is a large enough value of $n$ for which $E(L_{n, w})[p]\neq 0$. Since $L_{\op{cyc}}/K$ is a pro-$p$ extension, from Lemma \ref{NSW lemma}, it follows that 
    \[E(L_{n, w})[p]\neq 0\Rightarrow E(K_v)[p]\neq 0.\] Let $\ell$ be the prime number such that $w|\ell$. Then, the kernel of the reduction map
    \[E(K_v)\rightarrow \widetilde{E}(k_v)\] is a pro-$\ell$ group. Since $\ell\neq p$, it follows that there is an injection
    \[E(K_v)[p]\hookrightarrow \widetilde{E}(k_v)[p].\]This in turn implies that $\widetilde{E}(k_v)[p]\neq 0$, and hence, $v\in Q_2$.
\end{proof}

\begin{corollary}\label{rank vanishing corollary}
    Let $E_{/K}$ satisfy the conditions of Theorem \ref{Kida formula main thm}. Furthermore, assume that the only primes of $K$ that ramify in $L$ are in $Q_3$. Then, we have that 
    \[\lambda_p(E/L)=[L:K]\lambda_p(E/K). \]
\end{corollary}
\begin{proof}
    Let $w\in P_1\cup P_2$ and $v$ be the prime of $K$ such that $w|v$. Furthermore, assume that $e(w)>1$. From the formula \eqref{main lambda eqn}, it suffices to show that no such $w$ exists. Since $w\nmid p$, it follows that $v\nmid p$. Assume first that $v$ is a prime of bad reduction for $E$, i.e., $v\in Q_1$. Since $L/K$ is assumed to be unramified at all primes of $Q_1$ and $L_{\op{cyc}}/L$ is unramified at all primes $v\nmid p$, it follows that $v$ is unramified in the extension $L_{\op{cyc}}/K$. This implies that $e(w)=1$ and hence this case cannot arise. Therefore, $v$ must be a prime of good reduction and so must be $w$. Consequently, $w$ is not contained in $P_1$, and thus must be in $P_2$. Therefore by Lemma \ref{w in P_2 and Q_2} we find that $v\in Q_2$. Thus, $v\notin Q_3$ and must therefore be unramified in $L$. Hence $v$ is unramified in $L_{\op{cyc}}$. However, $e(w)>1$ and gives a contradiction to this. Therefore, no such prime $w$ can exist and we obtain that 
    \[\lambda_p(E/L)=[L:K]\lambda_p(E/K),\] since the local terms 
    \[\sum_{w\in P_i}(e(w)-1)=0\] for $i=1,2$. 
\end{proof}

\section{An Euler characteristic computation}\label{s 4}

\par Throughout this section $M$ will denote a cofinitely generated and cotorsion $\Lambda$-module. 
Consider the module of invariants $H^0(\Gamma, M)=M^\Gamma$ and module of co-invariants $H^1(\Gamma, M)=M_{\Gamma}=M/T M$. There is a natural map 
\[\phi_M: M^\Gamma\rightarrow M_\Gamma,\] that sends $x\in M^\Gamma$ to $x\mod{T M}$ in $M_\Gamma$. Since $M$ is cofinitely generated as a $\Lambda$-module, $M^\Gamma$ and $M_\Gamma$ are cofinitely generated as $\Z_p$-modules. Since $\Gamma\simeq \Z_p$ has cohomological dimension $1$, we have that $H^i(\Gamma, \cdot)=0$ for $i\geq 2$.

\begin{lemma}
    Letting $M$ be as above, we find that 
    \[\op{corank}_{\Z_p} M^\Gamma=\op{corank}_{\Z_p} M_\Gamma.\] 
\end{lemma}

\begin{proof}
    It follows from \cite[Theorem 1.1]{HowsonECs} that 
\[
\corank_{\Lambda}M= \op{corank}_{\Z_p} M^\Gamma-\corank_{\Z_p}M_\Gamma.
\]
Since $M$ is assumed to be cotorsion over $\Lambda$, the result follows.
\end{proof}

In particular, the above Lemma implies that $M^\Gamma$ is finite if and only if $M_\Gamma$ is finite. 

\begin{definition}\label{EC defn}
    Let $M$ be a cofinitely generated and cotorsion $\Lambda$-module. Then, we say that the \emph{Euler characteristic} of $M$ is well defined if $M^\Gamma$ (or equivalently) $M_\Gamma$ is finite. When this is the case, we define the \emph{Euler characteristic} of $M$ as follows
    \[\chi(\Gamma, M):=\prod_{i\geq 0} \left(H^i(\Gamma, M)\right)^{(-1)^i}=\left(\frac{\# M^\Gamma}{\# M_\Gamma}\right).\]
\end{definition}
Let $N$ denote the Pontryagin dual of $M$. Thus, $N$ is finitely generated and torsion as a $\Lambda$-module. Thus, $N$ is pseudo-isomorphic to $N'$ where 
\begin{equation}\label{N' structure theory} N'=\left(\bigoplus_{i=1}^s \Lambda/(p^{m_i})\right)\oplus \left(\bigoplus_{j=1}^t \Lambda/(f_j(T)^{n_j})\right),\end{equation}
cf. \eqref{structure decomposition M'}.
\begin{definition}\label{char series defn}
    Define the characteristic series as follows
\[f_M(T):=\prod_i p^{m_i}\times \prod_j f_j(T)^{n_j}\] and write 
\[f_M(T)=a_0+a_1T+a_2T^2+\dots+a_\lambda T^\lambda.\]
\end{definition}
Let $a$ and $b$ be $p$-adic numbers, we write $a\sim b$ to mean that there exists $u\in \Z_p^\times$ such that $a=u b$. 
\begin{proposition}\label{EC defined iff a_0 nonzero}
    Let $M$ be a cofinitely generated and cotorsion $\Lambda$-module. Then, with respect to notation above, the following conditions are equivalent.
    \begin{enumerate}
        \item The Euler characteristic $\chi(\Gamma, M)$ is well defined in the sense of Definition \ref{EC defn}.
        \item $a_0\neq 0$.
    \end{enumerate}
    Furthermore, if the above conditions are satisfied, then, $\chi(\Gamma, M)$ is an integer and
    \[a_0\sim \chi(\Gamma, M).\]
\end{proposition}
\begin{proof}
    Let $N$ denote the Pontryagin dual of $M$. We can identify the Pontryagin dual of $M^\Gamma$ with $N/T N$. Note that $N/T N$ is finite if and only if none of the distinguished polynomials $f_j(T)$ in \eqref{N' structure theory} are divisible by $T$. This in turn is equivalent to the condition that $T\nmid f_M(T)$, i.e., $a_0\neq 0$. Thus, we find that the Euler characteristic is well defined if and only if $a_0\neq 0$. It is easy to see that if $M$ is pseudo-isomorphic to $M'$, then, $\chi(\Gamma, M)=\chi(\Gamma, M')$. Therefore, we may assume without loss of generality that 
    \[N=\left(\bigoplus_{i=1}^s \frac{\Z_p\llbracket T\rrbracket}{(p^{m_i})}\right)\oplus \left( \bigoplus_{i=1}^s \frac{\Z_p\llbracket T\rrbracket}{(f_j(T)^{n_j})} \right).\] Then, it is easy to see from the decomposition above that $N[T]=0$. Therefore, we have that $M_\Gamma=\left(N[T]\right)^\vee=0$. On the other hand, \[M^\Gamma=\left(N/T N\right)^\vee\simeq \prod_i \Z/p^{m_i}\times \prod_j \Z/f_j(0)^{n_j}.\]
    Therefore, we find that 
    \[\chi(\Gamma, M)\sim f(0)=a_0.\]
\end{proof}

\begin{lemma}\label{mu and lambda zero criterion}
    Let $M$ be a cofinitely generated and cotorsion $\Lambda$-module for which the Euler characteristic is defined. Then the following assertions are equivalent
    \begin{enumerate}
        \item\label{p1 mu and lambda zero criterion} $\mu_p(M)=0$ and $\lambda_p(M)=0$, 
        \item\label{p2 mu and lambda zero criterion} $a_0$ is not divisible by $p$, 
        \item\label{p3 mu and lambda zero criterion} $\chi(\Gamma, M)=1$.
    \end{enumerate}
\end{lemma}
\begin{proof}
Let $N$ be the Pontryagin dual of $M$. Then, by the structure theorem for finitely generated and torsion $\Lambda$-modules, $N$ is pseudo-isomorphic to
$$N'\simeq \left(\bigoplus_{i=1}^s \Lambda/(p^{m_i})\right)\oplus \left(\bigoplus_{j=1}^t \Lambda/(f_j(T)^{n_j})\right),$$ as in \eqref{structure decomposition M'}.
\par First we show that the conditions \eqref{p1 mu and lambda zero criterion} and \eqref{p2 mu and lambda zero criterion} are equivalent. Assume that $\mu_p(M)=0$ and $\lambda_p(M)=0$. Therefore, $s=0$ and $t=0$ in the above decomposition, and consequently, $N'=0$. Therefore $f_M(T)=1$, and it thus follows that $a_0=1$. In particular, $a_0$ is not divisible by $p$. Conversely, suppose that $p\nmid a_0$. By definition, $f_M(T)=p^{\mu_p(M)} g_M(T)$, where $g_M(T)$ is the distinguished polynomial $\prod_j f_j(T)^{n_j}$. In particular, we find that 
\[a_0=p^{\mu_p(M)} g_M(0)\] is not divisible by $p$. This implies that $\mu_p(M)=0$ and $p\nmid g_M(0)$. However, $g_M(T)$ is distinguished and thus all its non-leading coefficients are divisible by $p$. The only possibility therefore is if $g_M(T)=1$. Recall that $\lambda_p(M)$ is the degree of $g_M(T)$, and we thus deduce that $\lambda_p(M)=0$. This shows that \eqref{p1 mu and lambda zero criterion} and \eqref{p2 mu and lambda zero criterion} are equivalent.

\par Since the Euler characteristic is defined, it follows from Lemma \ref{EC defined iff a_0 nonzero} that $\chi(\Gamma, M)$ is an integer and moreover, 
\[a_0\sim \chi(\Gamma, M).\]
Thus we deduce that 
\[p\nmid a_0\Leftrightarrow p\nmid \chi(\Gamma, M).\]
Also note that by definition, $\chi(\Gamma, M)$ is a power of $p$, and thus
\[p\nmid \chi(\Gamma, M)\Leftrightarrow \chi(\Gamma, M)=1.\]
This proves the equivalence of \eqref{p2 mu and lambda zero criterion} and \eqref{p3 mu and lambda zero criterion}. The proof is thus complete.
\end{proof}
Throughout the rest of this section we set $K:=\Q$ and impose the following assumption.

\begin{ass}\label{section 4 ass}
    Let $E$ be an elliptic curve over $\Q$ and assume that $E$ has additive reduction at $p$ and that there is an extension $L/\Q$ contained in $\Q(\mu_p)$ such that $E$ has good ordinary reduction at $\pi$, the prime of $L$ that lies above $p$.
\end{ass} 
Note that Assumption \ref{section 4 ass} is a special case of the more general Assumption \ref{reduction type}.
\begin{proposition}\label{cotorsion and EC well defined propn}
 Let $E_{/\Q}$ be an elliptic curve which satisfies Assumption \ref{section 4 ass}. Moreover, assume that $E$ has analytic rank $0$. Then, the following assertions hold.
 \begin{enumerate}
    \item The Selmer group $\op{Sel}_{p^\infty}(E/\Q_{\op{cyc}})$ is cotorsion over $\Lambda$.
     \item The Euler characteristic $\chi\left(\Gamma, \op{Sel}_{p^\infty}(E/\Q_{\op{cyc}})\right)$ is well defined. 
 \end{enumerate}
\end{proposition}
\begin{proof}
    That the Selmer group $\op{Sel}_{p^\infty}(E/\Q_{\op{cyc}})$ is cotorsion over $\Lambda$ follows from \cite[Theorem 3]{Delbs}. There is a natural map 
    \[\alpha: \op{Sel}_{p^\infty}(E/\Q)\rightarrow \op{Sel}_{p^\infty}(E/\Q_{\op{cyc}})^\Gamma\] which has finite kernel and cokernel (cf. \cite[p. 138 l.4 to p. 139 l. 11]{Delbs}). Since $E$ is assumed to have analytic rank $0$, it follows that $\op{Sel}_{p^\infty}(E/\Q)$ is finite. Since $\alpha$ has finite kernel and cokernel we deduce that $\op{Sel}_{p^\infty}(E/\Q_{\op{cyc}})^\Gamma$ is finite. Thus, the Euler characteristic of $\op{Sel}_{p^\infty}(E/\Q_{\op{cyc}})$ is defined. 
\end{proof}
We introduce some further notation. Recall that $E_{/L}$ has good reduction at $\pi$, let $\mathfrak{F}$ denote this reduction and
\[\Pi: E(L_\pi)\rightarrow \mathfrak{F}(\F_p)\] be the reduction map. The following Euler characteristic formula due to Delbourgo \cite{Delbs} will be used to construct elliptic curves $E_{/\Q}$ for which $\mu_p(E/\Q)$ and $\lambda_p(E/\Q)$ will both be $0$. Given a prime $\ell$, denote by $c_\ell(E)$ the Tamagawa number of $E$ at $\ell$.

\begin{theorem}[Euler characteristic formula]\label{ECF Delbourgo}
    Let $E_{/\Q}$ be an elliptic curve satisfying the following conditions 
    \begin{enumerate}
        \item Assumption \ref{section 4 ass} holds,
        \item $E$ has analytic rank $0$,
        \item $E(\Q)[p]=0$.
    \end{enumerate} Then, the Euler characteristic is given by 
    \[\chi(\Gamma, \op{Sel}_{p^\infty}(E/\Q_{\op{cyc}})\sim \# \Sh(E/\Q)[p^\infty] \times \# \mathfrak{F}(\F_p) \times \#\Pi\left(E(\Q_p)\right) \times \prod_{\ell\neq p} c_\ell(E) .\]
\end{theorem}
\begin{proof}
    The above result is due to Delbourgo, cf. \cite[p.148, l.-5 ]{Delbs}.
\end{proof}
The above formula gives us an explicit criterion for the vanishing of the $\mu$ and $\lambda$-invariants of $\op{Sel}_{p^\infty}(E/\Q_{\op{cyc}})$.
\begin{corollary}\label{equivalent cond for Delbs ECF}
    Let $E$ be an elliptic curve over $\Q$ satisfying the assumptions of Theorem \ref{ECF Delbourgo}. Then the following are equivalent.
    \begin{enumerate}
        \item $\lambda_p(E/\Q)=0$ and $\mu_p(E/\Q)=0$,
        \item $\Sh(E/\Q)[p^\infty]$ is trivial and $\# \mathfrak{F}(\F_p)$, $ \#\Pi\left(E(\Q_p)\right)$, $c_\ell(E)$ for $\ell \neq p$ are not divisible by $p$.
    \end{enumerate}
\end{corollary}
\begin{proof}
    By Theorem \ref{ECF Delbourgo} above, the Euler characteristic $\chi(\Gamma, \op{Sel}_{p^\infty}(E/\Q_{\op{cyc}}))=1$ if and only if $\Sh(E/\Q)[p^\infty]$ is trivial and none of $\# \mathfrak{F}(\F_p)$, $ \#\Pi\left(E(\Q_p)\right)$, $c_\ell(E)$ for $\ell \neq p$ are divisible by $p$. 
    Further, using Lemma \ref{mu and lambda zero criterion} the Euler characteristic $\chi(\Gamma, \op{Sel}_{p^\infty}(E/\Q_{\op{cyc}}))=1$ if and only if both $\lambda_p(E/\Q)=0$ and $\mu_p(E/\Q)=0$.\\
    This shows that the conditions are equivalent, and completes the proof. 
\end{proof}

\subsection*{An example}
We give an example of an elliptic curve $E/\Q$ for which the following conditions are satisfied.
\begin{enumerate}
    \item $E$ has bad additive reduction at $3$ and at the prime of $\Q(\mu_3)$ that lies above $3$, it has good ordinary reduction.
    \item The analytic rank of $E/\Q$ is zero and consequently, 
    $\op{Sel}_{p^\infty}(E/\Q)$ is finite.
    \item The Selmer group $\op{Sel}_{p^\infty}(E/\Q_{\op{cyc}})$ is cotorsion as a $\Lambda$-module.
    \item The invariants $\lambda_p(E/\Q)=0$ and $\mu_p(E/\Q)=0$.
\end{enumerate}
Consider the curve $E:y^2+y=x^3-3x-5$ over $\Q$. The conditions (1) and (2) above are checked on \href{https://www.lmfdb.org/EllipticCurve/Q/99/d/3}{LMFDB}. The condition (3) is satisfied due to part (1) of Proposition \ref{cotorsion and EC well defined propn}. Finally, the computations on \cite[p. 149 l. -9 to p. 150 l. 5]{Delbs} show that the condition (2) in Corollary \ref{equivalent cond for Delbs ECF} are satisfied. This implies that the condition (4) above is satisfied by $E$.
\par Let $L/\Q$ be the $\Z/3\Z$-extension of $\Q$ which is contained in $\Q(\mu_7)$. Then, the only prime that ramifies in $L$ is $7$. This is a prime of good reduction for $E$ and it can be checked that $\#\widetilde{E}(\F_7)=10$. Hence, $\widetilde{E}(\F_7)[3]=0$ and we deduce that $7\in Q_3$ (cf. Definition \ref{defn of the Q_is}). Thus by Corollary \ref{rank vanishing corollary}, it follows that 
\[\lambda_p(E/L)=3\lambda_p(E/\Q)=0.\]

\section{Density results}\label{s 5}

\par In this section we take $K=\Q$ and $L/\Q$ will always denote a Galois extension with Galois group $\op{Gal}(L/\Q)\simeq \Z/p\Z$. Set $\Delta_L$ to be the discriminant of $L/\Q$. We fix an elliptic curve $E_{/\Q}$.
\begin{ass}\label{last ass}
We assume throughout this section that the following assumptions are satisfied by $E$.
\begin{enumerate}
 \item There exists a finite Galois extension $K'/\Q$ with Galois group $\Delta:=\op{Gal}(K'/\Q)$ over which $E$ has good reduction. Moreover, assume that $p\nmid |\Delta|$.
        \item Let $S_{\op{add}}$ be the set of primes $v$ of $K$ not dividing $p$ at which $E$ has additive reduction. Then all primes of $S_{\op{add}}$ continue to have additive reduction in $L_{\op{cyc}}$ (this condition is automatically satisfied when $L/\Q$ is unramified at all primes of $S_{\op{add}}$ or if $p\geq 5$). 
        \item The Selmer group $\op{Sel}_{p^\infty}(E/\Q_{\op{cyc}})$ is cofinitely generated and cotorsion over $\Lambda$ with \[\mu_p(E/\Q)=0\text{ and }\lambda_p(E/\Q)=0.\]
\end{enumerate}
\end{ass}
The Corollary \ref{equivalent cond for Delbs ECF} gives us an explicit criterion for the above conditions to be satisfied. Recall that it was checked in the previous section that these conditions hold for the elliptic curve
\[E:y^2+y=x^3-3x-5.\]
Given a real number $X>0$, $\mathcal{S}(X)$ be the set of Galois extensions $L/\Q$ with $\op{Gal}(L/\Q)\simeq \Z/p\Z$ and such that $|\Delta_L|\leq X$. It follows from the Hermite-Minkowski theorem that the set $\mathcal{S}(X)$ is finite. Let $\mathcal{S}_E(X)$ be subset of $\mathcal{S}(X)$ consisting of the extensions for which the following conditions hold
\begin{itemize}
    \item $\op{Sel}_{p^\infty}(E/L)$ is cofinitely generated and cotorsion over $\Lambda$,
    \item $\mu_p(E/L)=0$ and $\lambda_p(E/L)=0$.
\end{itemize}
 We note that for $L\in \mathcal{S}_E(X)$, it follows from Proposition \ref{propn rank<=lambda} that $\op{rank} E(L)=0$. Thus, the rank remains stable in such extensions $L/\Q$. We prove asymptotic formulae for $N_E(X):=\#\mathcal{S}_E(X)$ respectively. First, we recall the Tauberian theorem of Delange, which will be applied to obtain our result. 
\begin{theorem}[Delange's Tauberian theorem]\label{tauberian thm}
    Let $f(s):=\sum_{n=1}^\infty a_n n^{-s}$ be a Dirichlet series with non-negative coefficients and $a>0$ be a real number. Assume that $f(s)$ converges for $\op{Re}(s)>a$ and has a meromorphic continuation to a neighbourhood $U$ of $\op{Re}(s)\geq a$. For $X>0$, we set 
    $g(X):=\sum_{n\leq X} a_n$. Assume that the only pole of $f(s)$ is at $s=a$ and the order of this pole is $b\in \mathbb{R}_{>0}$, i.e., 
    \[f(s)=\frac{1}{(s-a)^b} h(s)\] for some holomorphic function $h(s)$ defined on $U$. Then, there is a positive constant $c>0$ such that $X\rightarrow \infty$, we have that 
    \[g(X)\sim c X^a (\log X)^{b-1}.\]
\end{theorem}
\begin{proof}
    The result is a special case of \cite[Theorem 7.28]{tenenbaum2015introduction}. 
\end{proof}

Let $L/\Q$ be a Galois extension with $\op{Gal}(L/\Q)\simeq \Z/p\Z$ which is unramified at $p$. Let $\ell_1, \dots, \ell_k$ be primes that ramify in $L$. It follows from class field theory that each of the primes $\ell_i$ is $1\mod{p}$. Let $\ell_1, \dots, \ell_k$ be prime numbers that are $1\mod{p}$. Then, it is a straightforward exercise in class field theory that the number of $\Z/p\Z$-extensions $L/\Q$ that are ramified exactly at $\ell_1,\dots, \ell_k$ is equal to $(p-1)^{k-1}$. Moreover, since the primes $\ell_1, \dots, \ell_k$ are tamely ramified, it follows that $L/\Q$ has discriminant $\Delta_L=\left(\prod_{i=1}^k \ell_i\right)^{p-1}$. We take $\mathcal{Q}$ to denote the set of primes $\ell\in Q_3$ (cf. Definition \ref{defn of the Q_is}) such that $\ell\equiv 1\mod{p}$. Let $\alpha$ denote the natural density of $\mathcal{Q}$ and assume that $\alpha>0$. Let 
\[\rho_{E, p}:\op{Gal}(\bar{\Q}/\Q)\rightarrow \op{GL}_2(\F_p) \] be the Galois representation on $E[p]$.

\begin{lemma}\label{alpha value lemma}
    Suppose that the representation $\rho_{E,p}$ is surjective, then \[\alpha=\left(\frac{p^2 - p - 1}{p^3-p^2-p+1}\right).\]
\end{lemma}

\begin{proof}
    Recall that a prime $\ell\in \mathcal{Q}$ if $\ell\equiv 1\mod{p}$ and $\ell\in Q_3$. The set $Q_3$ consists of primes $\ell\neq p$ such that $E$ has good reduction at $\ell$ and $\widetilde{E}(\F_\ell)[p]=0$. Let $a_\ell(E)$ denote the Frobenius trace of $E$ at $\ell$, note that 
    \[a_\ell(E)=\ell+1-\# \widetilde{E}(\F_\ell).\] Since $\ell\equiv 1\mod{p}$, we find taht $a_\ell(E)\equiv 2-\# \widetilde{E}(\F_\ell)$. Thus, we find that \[\widetilde{E}(\F_\ell)[p]=0\Leftrightarrow a_\ell(E)\not \equiv 2\mod{p}.\] Let $\sigma_\ell$ denote the Frobenius at $\ell$. Since $\rho_{E, p}$ is unramified at $\ell$, there is a well defined matrix $\rho_{E,p}(\sigma_\ell)\in \op{GL}_2(\F_\ell)$. Moreover, it follows from the Weil pairing that $\op{det}\rho_{E,p}=\omega$, the mod-$p$ cyclotomic character. Therefore, we find that a prime $\ell\neq p$ of good reduction for $E$ is contained in $\mathcal{Q}$ if and only if 
    \[\op{trace} \rho_{E,p}(\sigma_\ell)\neq 2,\text{ and }\op{det}\rho_{E,p}(\sigma_\ell)=1.\]
    We take $\Q(E[p])$ to denote the Galois extension of $\Q$ which is fixed by the kernel of $\rho_{E,p}$. We have a natural isomorphism 
    \[\op{Gal}\left(\Q(E[p])/\Q\right)\xrightarrow{\sim}\op{image}(\rho_{E,p})=\op{GL}_2(\F_p).\]Thus one simply identifies $\op{Gal}\left(\Q(E[p])/\Q\right)$ with $\op{GL}_2(\F_p)$. With respect to this identification, $\ell\in \mathcal{Q}$ if and only if $\sigma_\ell$ is a matrix with trace $\neq 2$ and determinant $=1$. It then follows from the Chebotarev density theorem that 
    \[\alpha=\frac{\# \{A\in \op{SL}_2(\F_p)\mid \op{trace}(A)\neq 2\}}{\# \op{GL}_2(\F_p)}.\]
    We count the number of matrices of the form $A=\mtx{a}{b}{c}{2-a}$ in $\op{SL}_2(\F_p)$. This equals the number of triples $(a, b, c)\in \F_p^3$ such that $bc=a(2-a)-1=-a^2+2a-1=-(a-1)^2$. It is easy to see that the number of such triples is equal to $(p-1)^2+(2p-1)$. Hence, we have  
\begin{align*}
   & \# \{A\in \op{SL}_2(\F_p)\mid \op{trace}(A)\neq 2\} \\  = & \, \# \op{SL}_2(\F_p)-(p-1)^2-(2p-1)\\ = & p(p^2-1)-(p-1)^2-(2p-1) \\
= & p^3 - p^2 - p.
\end{align*}
Thus, we have shown that 
 \[\alpha=\left(\frac{p^3 - p^2 - p}{(p^2-p)(p^2-1)}\right)=\left(\frac{p^2 - p - 1}{p^3-p^2-p+1}\right).\]
\end{proof}

\begin{remark}
    Note that for our example from the previous section, $\rho_{E,3}$ is indeed surjective, as checked in \href{https://www.lmfdb.org/EllipticCurve/Q/99/d/3}{LMFDB}. 
\end{remark}

\par Given an integer of the form $n=\ell_1\dots \ell_k$, take $a_n:=(p-1)^{k-1}$ and set $a_n:=0$ otherwise. Then, we find from the discussion above that $a_n$ is the number of $\Z/p\Z$-extensions $L/\Q$ that are ramified at exactly the primes $\ell_1, \dots, \ell_k$ and have discriminant $n^{p-1}$. Setting $g(X):=\sum_{n\leq X} a_n$, we find that
\[\begin{split}g(X)=\#\{L\mid & \op{Gal}(L/\Q)\simeq \Z/p\Z, |\Delta_L|\leq X^{(p-1)},\\  & \text{ and }L \text{ is ramified only at primes in }\mathcal{Q}\}.\end{split} \]

\begin{proposition}
    With respect to the notation above, let $L/\Q$ be a $\Z/p\Z$ extension which is ramified only at a set of primes in $\mathcal{Q}$. Then, the following conditions hold
    \begin{enumerate}
    \item $\op{Sel}_{p^\infty}(E/L)$ is cofinitely generated and cotorsion over $\Lambda$,
    \item $\mu_p(E/L)=0$ and $\lambda_p(E/L)=0$,
    \item $\op{rank} E(L)=0$.
\end{enumerate}
\end{proposition}
Moreover, we find that
\begin{equation}\label{g leq ...}
    g(X)\leq N_E(X^{p-1}).
\end{equation}
\begin{proof}
    It follows from the assumptions on $E$ that the conditions of Theorem \ref{Kida formula main thm} are satisfied. Recall taht it is assumed that $\mu_p(E/\Q)=0$ and $\lambda_p(E/\Q)=0$. It follows from Theorem \ref{Kida formula main thm} that $\op{Sel}_{p^\infty}(E/L_{\op{cyc}})$ is cofinitely generated and cotorsion as a $\Lambda$-module and $\mu_p(E/L)=0$. Since $\mathcal{Q}$ is a subset of $Q_3$, we deduce from Corollary \ref{rank vanishing corollary} that $\lambda_p(E/L)=0$ as well, and thus the conditions (1) and (2) above are both satisfied. Finally, since 
    \[\op{rank} E(L)\leq \lambda_p(E/L)\] by Proposition \ref{propn rank<=lambda}, part (3) follows.
    That $g(X)\leq N_E(X^{p-1})$ simply follows as a consequence.
    \\ The result has thus been proven. 
\end{proof}
We apply Delange's Tauberian theorem (Theorem \ref{tauberian thm}) to prove an asymptotic lower bound for $g(X)$, and thus derive an asymptotic bound for $N_E(X)$. 

\begin{theorem}\label{last thm}
    With respect to notation above, we have that 
    \[N_E(X)\gg X^{\frac{1}{(p-1)}}(\log X)^{(p-1)\alpha-1},\] where $\alpha$ is the density of $\mathcal{Q}$. 
\end{theorem}
\begin{proof}
    The proof of the result follows along the same lines as \cite[Theorem 2.4]{serredivisibilite}. Nonetheless, we provide details for completeness. Set \[\begin{split}f(s)  :=\sum_{n=1}^\infty a_n n^{-s} & =\frac{1}{(p-1)}\sum_{T\subset \mathcal{Q}} (p-1)^{|T|} \left(\prod_{\ell \in T} \ell\right)^{-s} \\
     & = \frac{1}{(p-1)}\prod_{\ell\in \mathcal{Q}}\left(1+(p-1) \ell^{-s}\right).\end{split}\]
    It is easy to see that
\[\log f(s)=(p-1)\sum_{\ell\in \mathcal{Q}} \ell^{-s}+k_1(s),\] where $k_1(s)$ is holomorphic on $\op{Re}(s)\geq 1$; and as a consequence,
\[\log f(s) =(p-1)\alpha \log\left(\frac{1}{s-1}\right)+k_2(s), \] where $k_2(s)$ is holomorphic on $\op{Re}(s)\geq 1$. Thus, we deduce that 
\[f(s)=(s-1)^{-(p-1)\alpha}h(s),\] where $h(s)$ is a non-zero holomorphic function on $\op{Re}(s)\geq 1$. It follows from the Theorem \ref{tauberian thm} that 
\[g(X)\sim c X (\log X)^{(p-1)\alpha-1},\] where $c>0$ is a constant that does not depend on $X$. It then follows from \eqref{g leq ...} that 
 \[N_E(X)\gg X^{\frac{1}{(p-1)}}(\log X)^{(p-1)\alpha-1}.\]\\
 This completes the proof of the result. 
\end{proof}
\par We now give the proof of Theorem \ref{thm C}.
\begin{proof}[Proof of Theorem \ref{thm C}] 
\par Since it is assumed that $\rho_{E,p}$ is surjective, it follows from Lemma \ref{alpha value lemma} that \[\alpha=\left(\frac{p^2 - p - 1}{p^3-p^2-p+1}\right).\] Noting that $-\beta=\alpha(p-1)-1$, the result is thus follows as a direct consequence of Theorem \ref{last thm}.
\end{proof}

\subsection*{Declarations}
\begin{description}
    \item[Ethics approval and consent to participate] Not applicable, no animals were studied or involved in the preparation of this manuscript.
    \item[Consent for publication] The authors give their consent for the publication of the manuscript under review.
    \item[Availability of data and materials] No data was generated or analyzed in obtaining the results in this article.
    \item[Competing interests] There are no competing interests to report.
    \item[Funding] There are no funding sources to report.
    \item[Authors' contributions] Both authors contributed equally in the preparation of this manuscript.
    \item[Acknowledgements] There is no one that the authors wish to acknowledge.
\end{description}

\bibliographystyle{alpha}
\bibliography{references}

\begin{thebibliography}{NSW08}

\bibitem[CG96]{CoatesGreenberg}
J.~Coates and R.~Greenberg.
\newblock Kummer theory for abelian varieties over local fields.
\newblock {\em Invent. Math.}, 124(1-3):129--174, 1996.

\bibitem[Del98]{Delbs}
Daniel Delbourgo.
\newblock Iwasawa theory for elliptic curves at unstable primes.
\newblock {\em Compositio Math.}, 113(2):123--153, 1998.

\bibitem[DFK04]{DFK1}
Chantal David, Jack Fearnley, and Hershy Kisilevsky.
\newblock On the vanishing of twisted {$L$}-functions of elliptic curves.
\newblock {\em Experiment. Math.}, 13(2):185--198, 2004.

\bibitem[DFK07]{DFK2}
Chantal David, Jack Fearnley, and Hershy Kisilevsky.
\newblock Vanishing of {$L$}-functions of elliptic curves over number fields.
\newblock In {\em Ranks of elliptic curves and random matrix theory}, volume 341 of {\em London Math. Soc. Lecture Note Ser.}, pages 247--259. Cambridge Univ. Press, Cambridge, 2007.

\bibitem[FW79]{ferrerolawrence}
Bruce Ferrero and Lawrence~C. Washington.
\newblock The {I}wasawa invariant {$\mu _{p}$} vanishes for abelian number fields.
\newblock {\em Ann. of Math. (2)}, 109(2):377--395, 1979.

\bibitem[Gre99]{greenberg}
Ralph Greenberg.
\newblock Iwasawa theory for elliptic curves.
\newblock In {\em Arithmetic theory of elliptic curves ({C}etraro, 1997)}, volume 1716 of {\em Lecture Notes in Math.}, pages 51--144. Springer, Berlin, 1999.

\bibitem[Gre01]{Greenbergintro}
Ralph Greenberg.
\newblock Introduction to {I}wasawa theory for elliptic curves.
\newblock In {\em Arithmetic algebraic geometry ({P}ark {C}ity, {UT}, 1999)}, volume~9 of {\em IAS/Park City Math. Ser.}, pages 407--464. Amer. Math. Soc., Providence, RI, 2001.

\bibitem[HM99]{HachimoriMatsuno}
Yoshitaka Hachimori and Kazuo Matsuno.
\newblock An analogue of {K}ida's formula for the {S}elmer groups of elliptic curves.
\newblock {\em J. Algebraic Geom.}, 8(3):581--601, 1999.

\bibitem[How02]{HowsonECs}
Susan Howson.
\newblock Euler characteristics as invariants of {I}wasawa modules.
\newblock {\em Proc. London Math. Soc. (3)}, 85(3):634--658, 2002.

\bibitem[Ima75]{Imai}
Hideo Imai.
\newblock A remark on the rational points of abelian varieties with values in cyclotomic {$\mathbb{Z}_{p}$}-extensions.
\newblock {\em Proc. Japan Acad.}, 51:12--16, 1975.

\bibitem[Iwa73]{Iwasawareference}
Kenkichi Iwasawa.
\newblock On {${\bf Z}_{l}$}-extensions of algebraic number fields.
\newblock {\em Ann. of Math. (2)}, 98:246--326, 1973.

\bibitem[Iwa81]{Iwa81}
Kenkichi Iwasawa.
\newblock {Riemann-Hurwitz} formula and $p$-adic {Galois} representations for number fields.
\newblock {\em Tohoku Math. J., Second Series}, 33(2):263--288, 1981.

\bibitem[Kat04]{Katomodforms}
Kazuya Kato.
\newblock {$p$}-adic {H}odge theory and values of zeta functions of modular forms.
\newblock Number 295, pages ix, 117--290. 2004.
\newblock Cohomologies $p$-adiques et applications arithm\'{e}tiques. III.

\bibitem[Kid80]{Kid80}
Y{\^u}ji Kida.
\newblock $\ell$-extensions of {CM}-fields and cyclotomic invariants.
\newblock {\em J. Number Theory}, 12:519--528, 1980.

\bibitem[KR22]{AnweshDebs}
Debanjana Kundu and Anwesh Ray.
\newblock Iwasawa invariants for elliptic curves over {$\Bbb{Z}_p$}-extensions and {K}ida's formula.
\newblock {\em Forum Math.}, 34(4):945--967, 2022.

\bibitem[M\"85]{Maki}
Sirpa M\"{a}ki.
\newblock On the density of abelian number fields.
\newblock {\em Ann. Acad. Sci. Fenn. Ser. A I Math. Dissertationes}, (54):104, 1985.

\bibitem[Mal02]{malle2002distribution}
Gunter Malle.
\newblock On the distribution of {G}alois groups.
\newblock {\em Journal of Number Theory}, 92(2):315--329, 2002.

\bibitem[Mal04]{malle2004distribution}
Gunter Malle.
\newblock On the distribution of {G}alois groups, {II}.
\newblock {\em Experimental Mathematics}, 13(2):129--135, 2004.

\bibitem[Maz72]{Maz72}
Barry Mazur.
\newblock Rational points of abelian varieties with values in towers of number fields.
\newblock {\em Invent. Math.}, 18(3-4):183--266, 1972.

\bibitem[MR18]{mazurrubin}
Barry Mazur and Karl Rubin.
\newblock Diophantine stability.
\newblock {\em Amer. J. Math.}, 140(3):571--616, 2018.
\newblock With an appendix by Michael Larsen.

\bibitem[NSW08]{NSW}
J\"{u}rgen Neukirch, Alexander Schmidt, and Kay Wingberg.
\newblock {\em Cohomology of number fields}, volume 323 of {\em Grundlehren der mathematischen Wissenschaften [Fundamental Principles of Mathematical Sciences]}.
\newblock Springer-Verlag, Berlin, second edition, 2008.

\bibitem[Ray23]{Anweshdioph}
Anwesh Ray.
\newblock Arithmetic statistics and diophantine stability for elliptic curves.
\newblock {\em Ramanujan J.}, 62(1):215--239, 2023.

\bibitem[Ser75]{serredivisibilite}
Jean-Pierre Serre.
\newblock Divisibilit\'{e} de certaines fonctions arithm\'{e}tiques.
\newblock In {\em S\'{e}minaire {D}elange-{P}isot-{P}oitou (16e ann\'{e}e: 1974/75), {T}h\'{e}orie des nombres, {F}asc. 1}, pages Exp. No. 20, 28. Secr\'{e}tariat Math., Paris, 1975.

\bibitem[Ten15]{tenenbaum2015introduction}
G{\'e}rald Tenenbaum.
\newblock {\em Introduction to analytic and probabilistic number theory}, volume 163.
\newblock American Mathematical Soc., 2015.

\bibitem[Was97]{washingtoncyclotomicfields}
Lawrence~C. Washington.
\newblock {\em Introduction to cyclotomic fields}, volume~83 of {\em Graduate Texts in Mathematics}.
\newblock Springer-Verlag, New York, second edition, 1997.

\bibitem[Wri89]{Wright}
David~J. Wright.
\newblock Distribution of discriminants of abelian extensions.
\newblock {\em Proc. London Math. Soc. (3)}, 58(1):17--50, 1989.

\end{thebibliography}
\end{document}